\documentclass[12pt,reqno,draft]{amsart} 
\usepackage{amssymb,amscd,url}

\begin{document}

\allowdisplaybreaks

\newif\ifdraft 
\drafttrue
\newcommand{\DRAFTNUMBER}{1}
\newcommand{\DATE}{\today\ \ifdraft(Draft \DRAFTNUMBER)\fi}
\newcommand{\TITLE}{Variation of Periods Modulo $p$ in Arithmetic Dynamics}
\newcommand{\TITLERUNNING}{Periods Modulo $p$ in Arithmetic Dynamics}



\newtheorem{theorem}{Theorem}
\newtheorem{lemma}[theorem]{Lemma}
\newtheorem{conjecture}[theorem]{Conjecture}
\newtheorem{proposition}[theorem]{Proposition}
\newtheorem{corollary}[theorem]{Corollary}
\newtheorem*{claim}{Claim}

\theoremstyle{definition}
\newtheorem{question}{Question}
\renewcommand{\thequestion}{\Alph{question}} 
\newtheorem*{definition}{Definition}
\newtheorem{example}[theorem]{Example}

\theoremstyle{remark}
\newtheorem{remark}[theorem]{Remark}
\newtheorem*{acknowledgement}{Acknowledgements}



\newenvironment{notation}[0]{%
  \begin{list}%
    {}%
    {\setlength{\itemindent}{0pt}
     \setlength{\labelwidth}{4\parindent}
     \setlength{\labelsep}{\parindent}
     \setlength{\leftmargin}{5\parindent}
     \setlength{\itemsep}{0pt}
     }%
   }%
  {\end{list}}

\newenvironment{parts}[0]{%
  \begin{list}{}%
    {\setlength{\itemindent}{0pt}
     \setlength{\labelwidth}{1.5\parindent}
     \setlength{\labelsep}{.5\parindent}
     \setlength{\leftmargin}{2\parindent}
     \setlength{\itemsep}{0pt}
     }%
   }%
  {\end{list}}
\newcommand{\Part}[1]{\item[\upshape#1]}

\renewcommand{\a}{\alpha}
\renewcommand{\b}{\beta}
\newcommand{\g}{\gamma}
\renewcommand{\d}{\delta}
\newcommand{\e}{\epsilon}
\newcommand{\f}{\varphi}
\newcommand{\bfphi}{{\boldsymbol{\f}}}
\renewcommand{\l}{\lambda}
\renewcommand{\k}{\kappa}
\newcommand{\lhat}{\hat\lambda}
\newcommand{\m}{\mu}
\newcommand{\bfmu}{{\boldsymbol{\mu}}}
\renewcommand{\o}{\omega}
\renewcommand{\r}{\rho}
\newcommand{\rbar}{{\bar\rho}}
\newcommand{\s}{\sigma}
\newcommand{\sbar}{{\bar\sigma}}
\renewcommand{\t}{\tau}
\newcommand{\z}{\zeta}

\newcommand{\D}{\Delta}
\newcommand{\G}{\Gamma}
\newcommand{\F}{\Phi}

\newcommand{\ga}{{\mathfrak{a}}}
\newcommand{\gA}{{\mathfrak{A}}}
\newcommand{\gb}{{\mathfrak{b}}}
\newcommand{\gB}{{\mathfrak{B}}}
\newcommand{\gC}{{\mathfrak{C}}}
\newcommand{\gD}{{\mathfrak{D}}}
\newcommand{\gm}{{\mathfrak{m}}}
\newcommand{\gn}{{\mathfrak{n}}}
\newcommand{\go}{{\mathfrak{o}}}
\newcommand{\gO}{{\mathfrak{O}}}
\newcommand{\gp}{{\mathfrak{p}}}
\newcommand{\gP}{{\mathfrak{P}}}
\newcommand{\gq}{{\mathfrak{q}}}
\newcommand{\gR}{{\mathfrak{R}}}

\newcommand{\Abar}{{\bar A}}
\newcommand{\Ebar}{{\bar E}}
\newcommand{\Kbar}{{\bar K}}
\newcommand{\Pbar}{{\bar P}}
\newcommand{\Sbar}{{\bar S}}
\newcommand{\Tbar}{{\bar T}}
\newcommand{\ybar}{{\bar y}}
\newcommand{\phibar}{{\bar\f}}

\newcommand{\Acal}{{\mathcal A}}
\newcommand{\Bcal}{{\mathcal B}}
\newcommand{\Ccal}{{\mathcal C}}
\newcommand{\Dcal}{{\mathcal D}}
\newcommand{\Ecal}{{\mathcal E}}
\newcommand{\Fcal}{{\mathcal F}}
\newcommand{\Gcal}{{\mathcal G}}
\newcommand{\Hcal}{{\mathcal H}}
\newcommand{\Ical}{{\mathcal I}}
\newcommand{\Jcal}{{\mathcal J}}
\newcommand{\Kcal}{{\mathcal K}}
\newcommand{\Lcal}{{\mathcal L}}
\newcommand{\Mcal}{{\mathcal M}}
\newcommand{\Ncal}{{\mathcal N}}
\newcommand{\Ocal}{{\mathcal O}}
\newcommand{\Pcal}{{\mathcal P}}
\newcommand{\Qcal}{{\mathcal Q}}
\newcommand{\Rcal}{{\mathcal R}}
\newcommand{\Scal}{{\mathcal S}}
\newcommand{\Tcal}{{\mathcal T}}
\newcommand{\Ucal}{{\mathcal U}}
\newcommand{\Vcal}{{\mathcal V}}
\newcommand{\Wcal}{{\mathcal W}}
\newcommand{\Xcal}{{\mathcal X}}
\newcommand{\Ycal}{{\mathcal Y}}
\newcommand{\Zcal}{{\mathcal Z}}

\renewcommand{\AA}{\mathbb{A}}
\newcommand{\BB}{\mathbb{B}}
\newcommand{\CC}{\mathbb{C}}
\newcommand{\FF}{\mathbb{F}}
\newcommand{\GG}{\mathbb{G}}
\newcommand{\NN}{\mathbb{N}}
\newcommand{\PP}{\mathbb{P}}
\newcommand{\QQ}{\mathbb{Q}}
\newcommand{\RR}{\mathbb{R}}
\newcommand{\ZZ}{\mathbb{Z}}

\newcommand{\bfa}{{\mathbf a}}
\newcommand{\bfb}{{\mathbf b}}
\newcommand{\bfc}{{\mathbf c}}
\newcommand{\bfe}{{\mathbf e}}
\newcommand{\bff}{{\mathbf f}}
\newcommand{\bfg}{{\mathbf g}}
\newcommand{\bfp}{{\mathbf p}}
\newcommand{\bfr}{{\mathbf r}}
\newcommand{\bfs}{{\mathbf s}}
\newcommand{\bft}{{\mathbf t}}
\newcommand{\bfu}{{\mathbf u}}
\newcommand{\bfv}{{\mathbf v}}
\newcommand{\bfw}{{\mathbf w}}
\newcommand{\bfx}{{\mathbf x}}
\newcommand{\bfy}{{\mathbf y}}
\newcommand{\bfz}{{\mathbf z}}
\newcommand{\bfA}{{\mathbf A}}
\newcommand{\bfF}{{\mathbf F}}
\newcommand{\bfB}{{\mathbf B}}
\newcommand{\bfD}{{\mathbf D}}
\newcommand{\bfG}{{\mathbf G}}
\newcommand{\bfI}{{\mathbf I}}
\newcommand{\bfM}{{\mathbf M}}
\newcommand{\bfzero}{{\boldsymbol{0}}}

\newcommand{\Adele}{\textsf{\upshape A}}
\newcommand{\Aut}{\operatorname{Aut}}
\newcommand{\Br}{\operatorname{Br}}  
\newcommand{\Disc}{\operatorname{Disc}}
\newcommand{\density}{{\boldsymbol\delta}}
\newcommand{\densitysup}{\overline{\density}}
\newcommand{\densityinf}{\underline{\density}}
\newcommand{\Div}{\operatorname{Div}}
\newcommand{\End}{\operatorname{End}}
\newcommand{\Fbar}{{\bar{F}}}
\newcommand{\FOD}{\textup{FOM}}
\newcommand{\FOM}{\textup{FOD}}
\newcommand{\Gal}{\operatorname{Gal}}
\newcommand{\GL}{\operatorname{GL}}
\newcommand{\Index}{\operatorname{Index}}
\newcommand{\Image}{\operatorname{Image}}
\newcommand{\liftable}{{\textup{liftable}}}
\newcommand{\hhat}{{\hat h}}
\newcommand{\Ksep}{K^{\textup{sep}}}
\newcommand{\Ker}{{\operatorname{ker}}}
\newcommand{\Lsep}{L^{\textup{sep}}}
\newcommand{\Lift}{\operatorname{Lift}}
\newcommand{\pp}{\operatorname{pp}}  
\newcommand{\vlim}{\operatornamewithlimits{\text{$v$}-lim}}
\newcommand{\wlim}{\operatornamewithlimits{\text{$w$}-lim}}
\newcommand{\MOD}[1]{~(\textup{mod}~#1)}
\newcommand{\Norm}{{\operatorname{\mathsf{N}}}}
\newcommand{\notdivide}{\nmid}
\newcommand{\normalsubgroup}{\triangleleft}
\newcommand{\odd}{{\operatorname{odd}}}
\newcommand{\onto}{\twoheadrightarrow}
\newcommand{\Orbit}{\mathcal{O}}
\newcommand{\ord}{\operatorname{ord}}
\newcommand{\Per}{\operatorname{Per}}
\newcommand{\PrePer}{\operatorname{PrePer}}
\newcommand{\PGL}{\operatorname{PGL}}
\newcommand{\Pic}{\operatorname{Pic}}
\newcommand{\Prob}{\operatorname{Prob}}
\newcommand{\Qbar}{{\bar{\QQ}}}
\newcommand{\rank}{\operatorname{rank}}
\newcommand{\Resultant}{\operatorname{Res}}
\renewcommand{\setminus}{\smallsetminus}
\newcommand{\SL}{\operatorname{SL}}
\newcommand{\Span}{\operatorname{Span}}
\newcommand{\tors}{{\textup{tors}}}
\newcommand{\Trace}{\operatorname{Trace}}
\newcommand{\twistedtimes}{\mathbin{%
   \mbox{$\vrule height 6pt depth0pt width.5pt\hspace{-2.2pt}\times$}}}
\newcommand{\UHP}{{\mathfrak{h}}}    
\newcommand{\Wreath}{\operatorname{Wreath}}
\newcommand{\<}{\langle}
\renewcommand{\>}{\rangle}

\newcommand{\longhookrightarrow}{\lhook\joinrel\longrightarrow}
\newcommand{\longonto}{\relbar\joinrel\twoheadrightarrow}

\newcommand{\Spec}{\operatorname{Spec}}
\renewcommand{\div}{{\operatorname{div}}}

\newcounter{CaseCount}
\Alph{CaseCount}
\def\Case#1{\par\vspace{1\jot}\noindent
\stepcounter{CaseCount}
\framebox{Case \Alph{CaseCount}.\enspace#1}
\par\vspace{1\jot}\noindent\ignorespaces}

\title[\TITLERUNNING]{\TITLE}
\date{\DATE}

\author[Joseph H. Silverman]{Joseph H. Silverman}
\email{jhs@math.brown.edu}
\address{Mathematics Department, Box 1917
         Brown University, Providence, RI 02912 USA}
\subjclass{Primary: 11G35; Secondary:  11B37, 14G40, 37F10}
\keywords{arithmetic dynamical systems, orbit modulo $p$}
\thanks{The author's research supported by NSF grant DMS-0650017}

\begin{abstract}
Let $\f:V\to V$ be a self-morphism of a quasiprojective variety defined
over a number field~$K$ and let~$P\in V(K)$ be a point with infinite
orbit under iteration of~$\f$. For each prime~$\gp$ of good reduction,
let~$m_\gp(\f,P)$ be the size of the $\f$-orbit of the reduction of~$P$
modulo~$\gp$. Fix any~$\e>0$. We show that for almost all
primes~$\gp$ in the sense of analytic density, the orbit
size~$m_\gp(\f,P)$ is larger than~$(\log\Norm_{K/\QQ}\gp)^{1-\e}$.
\end{abstract}


\maketitle

\section*{Introduction}
Let
\[
  \f:\PP^N_\QQ\longrightarrow\PP^N_\QQ
\]
be a morphism of degree~$d$ defined over~$\QQ$ and let~$P\in\PP^N(\QQ)$
be a point with infinite forward orbit
\[
  \Ocal_\f(P) = \bigl\{P,\f(P),\f^2(P),\dots\bigr\}.
\]
For all but finitely many primes~$p$, we can reduce~$\f$ to obtain a
morphism
\[
  \tilde\f_p : \PP^N_{\FF_p}\longrightarrow\PP^N_{\FF_p}
\]
whose degree is still~$d$. We write~$m_p(\f,P)$ for the
size of the orbit of the reduced point~$\tilde P = P\bmod p$,
\[
  m_p(\f,P) = \#\Ocal_{\tilde\f_p}(\tilde P).
\]
(For the remaining primes we define~$m_p(\f,P)$ to be~$\infty$.)
\par
Using an elementary height argument (see
Corollary~\ref{cor:mpgecloglogp}), one can show that
\[
  m_p(\f,P) \ge d\log\log p + O(1)
  \quad\text{for all~$p$,}
\]
but this is a very weak lower bound for the size of the mod~$p$
orbits. Our principal results say that for most primes~$p$, we can do
(almost) exponentially better. In the following result, we
write~$\density(\Pcal)$ for the logarithmic analytic density of a set
of primes~$\Pcal$. (See Section~\ref{section:setup} for the precise
definition of~$\density$ and the associated lower
density~$\densityinf$.)

\begin{theorem}
\label{thm:densityQ}
With notation as above, we have the following\textup{:}
\begin{parts}
\Part{(a)}
For all~$\g<1$,
\[
  \density\bigl\{p : m_p(\f,P)\ge(\log p)^\g\bigr\} = 1.
\]
\Part{(b)}
There is a constant~$C=C(N,\f,P)$ so that for all~$\e>0$,
\[
  \densityinf\bigl\{p : m_p(\f,P)\ge \e\log p\bigr\} \ge 1-C\e.
\]
\end{parts}
\end{theorem}

More generally, we prove analogous results for any
self-mor\-phism $\f:V\to V$ of a quasiprojective variety~$V$ defined
over a number field~$K$. See Section~\ref{section:setup} for the basic
setup and Theorem~\ref{thm:nfdensity} for the precise statement of
our main result.
\par
The proof of Theorem~\ref{thm:densityQ}, and its generalization
Theorem~\ref{thm:nfdensity}, proceeds in two steps. In the first step
we prove that there is an integer~$D(m)$ satisfying~$\log\log D(m)\ll{m}$ 
with the property that
\[
  \text{$m_p(\f,P)\le m$ \quad if and only if\quad $p|D(m)$.}
\]
This is done using a height estimate for rational maps
(Proposition~\ref{prop:hfnQlednhQC}) and a height estimate for
arithmetic distances (Proposition~\ref{prop:gDgAgBhAhB}).  The second
part of the proof uses the first part to prove an analytic estimate
(Theorem~\ref{thm:analyticresult}) of the following form: for
all~$\l\ge1$ there is a constant~$C=C(\f,\l)$ so that
\begin{equation}
  \label{eqn:sumQ}
  \sum_{p~\text{prime}} \frac{\log p}{pe^{sm_p(\f,P)^\l}}
  \le \frac{C}{s^{1/\l}}
  \qquad\text{for all $s>0$.}
\end{equation}
The inequality~\eqref{eqn:sumQ} is a dynamical analogue of the results
in~\cite{MR1395936}, which treated the case of periods modulo~$p$ of
points in algebraic groups.
\par
Theorem~\ref{thm:densityQ} says that for most~$p$, the mod~$p$ orbit
of~$P$ has size (almost) as large as~$\log p$. If~$\f$ were a random
map, we would expect most orbits to have size on the order
of~$\sqrt{\#\PP^N(\FF_p)}\approx p^{N/2}$. In
Section~\ref{section:conjandexp} we do some experiments using quadratic
polynomials~$\f_c(z)=z^2+c$ which seem to suggest that~$m_p(\f_c,\a)$
grows slightly slower than~$\sqrt{p}$.

\section{Notation and Statement of Main Result}
\label{section:setup}

In this section we set notation for our basic objects of study, give
some basic definitions, and state our main result. We start with the
dynamical setup.

\begin{definition}
Let~$K/\QQ$ be a number field, let~$V\subset\PP^N_K$ be a quasiprojective
variety defined over~$K$, and let
\[
  \f:V\longrightarrow V
\]
be a morphism defined over~$K$. Fix a point~$P\in V(K)$ whose
forward~$\f$-orbit
\[
  \Orbit_\f(P) = \bigl\{P,\f(P),\f^2(P),\dots\bigr\}
\]
is infinite. For each prime ideal~$\gp$ of~$K$ at which~$V$ and~$\f$
have good reduction, let
\[
  m_\gp = m_\gp(\f,P)
   = \text{size of the $\tilde\f$-orbit of~$\tilde P$ in $\tilde V(\FF_\gp)$}.
\]
If~$V$ has bad reduction at~$\gp$, we set~$m_\gp=\infty$. 
\end{definition}

\begin{remark}
The question of good versus bad reduction requires choosing models
for~$V$ and~$\f$ over the ring of integers of~$K$. However, different
choices affect only finitely many of the~$m_\gp(\f,,P)$ values, and
thus have no effect on our density results. We thus assume throughout
that a particular model has been fixed.
\end{remark}

We next define the analytic density that will be used in the statement
of our main result.

\begin{definition}
Let~$K/\QQ$ be a number field with
ring of integers~$R_K$.
For any set of primes~$\Pcal\subset\Spec(R_K)$, define
the partial~$\zeta$-function for~$\Pcal$ by
\[
  \z_K(\Pcal,s) 
   = \prod_{\gp\in\Pcal}\left(1-\frac{1}{\Norm_{K/\QQ}\gp^s}\right)^{-1}.
\]
This Euler product defines an analytic function 
on~$\operatorname{Re}(s)>1$. As usual,
we write~$\z_K(s)$ for the $\z$-function of the field~$K$.
Then the  (\emph{logarithmic analytic}) \emph{density} of~$\Pcal$ 
is given by the following limit, assuming that the limit exists:
\[
  \density(\Pcal) = \lim_{s\to1^+} \frac{d\log\z_K(\Pcal,s)}{d\log\z_K(s)}
  = \lim_{s\to1^+} \frac{\z'_K(\Pcal,s)/\z_K(\Pcal,s)}{\z'_K(s)/\z_K(s)}.
\]
Expanding the logarithm before differentiating and using the fact
that $\z_K(s)$ has a simple pole at~$s=1$, it is easy to check that
the density is also given by the formula
\[
  \density(\Pcal) 
  = \lim_{s\to 1^+} (s-1)\sum_{\gp\in\Pcal} 
       \frac{\log\Norm_{K/\QQ}\gp}{\Norm_{K/\QQ}\gp^s}.
\]
We similarly define upper and lower densities~$\densitysup(\Pcal)$ 
and~$\densityinf(\Pcal)$ by
replacing the limit with the limsup or the liminf, respectively.
\end{definition}

With this notation, we can now state our main result.

\begin{theorem}
\label{thm:nfdensity}
Let~$K/\QQ$ be a number field
and let~$\f:V/K\to V/K$, and~$P\in V(K)$ be as described in
this section. Further, let~$m_\gp(\f,P)$ denote the size
of the~$\tilde\f$-orbit of~$\tilde P$ in~$\tilde V(\FF_\gp)$.
\begin{parts}
\Part{(a)}
For all~$\g<1$ we have
\[
  \density\bigl\{\gp\in\Spec R_K:m_p(\f,P)\ge(\log \Norm\gp)^\g\bigr\} = 1.
\]
\Part{(b)}
There is a constant~$C=C(K,V,\f,P)$ so that for all~$\e>0$,
\[
  \densityinf\bigl\{\gp\in\Spec R_K:m_p(\f,P)\ge \e\log\Norm\gp\bigr\} 
        \ge 1 - C \e.
\]
\end{parts}
\end{theorem}

\section{Height and norm estimates}
\label{section:htandnorm}

In this section we prove various estimates for heights and norms that
will be needed for the proof of our main result.  To ease notation,
for the remainder of this paper we fix the number field~$K/\QQ$ and
write~$\Norm\ga$ for the~$K/\QQ$ norm of a fractional ideal~$\ga$
of~$K$.

\begin{proposition}
\label{prop:hfnQlednhQC}
With notation as in Section~$\ref{section:setup}$ and
Theorem~$\ref{thm:nfdensity}$, there are constants
\[
  d=d(V,\f)\ge2\quad\text{and}\quad C=C(V,\f)\ge0
\]
so that
\[
  h\bigl(\f^n(Q)\bigr) \le d^n\bigl(h(Q) + C)
  \quad\text{for all $n\ge0$ and all $Q\in V(\Kbar)$.}
\]
\end{proposition}
\begin{proof}
We are given that~$\f$ is a morphism on~$V$, but note that~$V$ is only
quasiprojective, i.e.,~$V$ is a Zariski open subset of a Zariski
closed subset of~$\PP^N$. We write~$V$ as a union of open
subsets~$V_1,\ldots,V_t$ such that on each~$V_i$ we can write
\[
  \f_i = \f|_{V_i} = [F_{i0},F_{i1},\ldots,F_{iN}],
\]
where the~$F_{ij}$ are homogeneous polynomials and such that
\[
  F_{i0},F_{i1},\dots,F_{iN}
  \quad\text{do not simultaneously vanish on~$V_i$.}
\]
We may view~$\f_i$ as a rational map $\f_i:\PP^N\dashrightarrow\PP^N$
of degree~$d_i=\deg F_{ij}$. Letting~$Z_i\subset\PP^N$ be the locus of
indeterminacy for the rational map~$\f_i$, we have the elementary
height estimate
\begin{equation}
  \label{eqn:hfiQledihQC}
  h\bigl(\f_i(Q)\bigr) 
  \le d_i h(Q) + C(\f_i),
  \quad\text{valid for all $Q\in\PP^N(\Kbar)\setminus Z_i$.}
\end{equation}
(See~\cite[Theorem~B.2.5(a)]{MR1745599}.) By construction,
\[
  V \cap Z_1\cap Z_2\cap\cdots\cap Z_t = \emptyset,
\]
so for all~$Q\in V(\Kbar)$ we obtain the inequality
\begin{align*}
  h\bigl(\f(Q)\bigr)
  &= h\bigl(\f_i(Q)\bigr)
    &&\text{for any $i$ with $Q\notin Z_i$,} \\
  &\le \max_{\text{$i$ with $Q\notin Z_i$}} d_i h(Q) + C(\f_i) 
    &&\text{from \eqref{eqn:hfiQledihQC},} \\
  &\le \max_{1\le i\le n} d_i h(Q) + \max_{1\le i\le n} C(\f_i).
\end{align*}
Setting~
\[
  d=\max\{2,d_1,\dots,d_t\}
  \quad\text{and}\quad
  C = \max\bigl\{C(\f_1),\ldots,C(\f_t)\bigr\},
\]
we have
\[
  h\bigl(\f(Q)\bigr) \le d h(Q) + C
  \quad\text{for all $Q\in V(\Kbar)$.}
\]
Applying this iteratively yields
\begin{equation}
  \label{eqn:hfnQdnhQ}
  h\bigl(\f^n(Q)\bigr) \le d^n h(Q) + (1+d+\dots+d^{n-1}) C
  \le d^n\bigl(h(Q)+C),
\end{equation}
(note that $d\ge2$ by assumption) which is the desired result.
\end{proof}

\begin{remark}
If $\f:\PP^N\to\PP^N$ is a finite morphism, then in the statement of
Proposition~\ref{prop:hfnQlednhQC} we can
take~$d=\max\{\deg\f,2\}$. More precisely, in this situation a
standard property of height functions~\cite[B.2.5(b)]{MR1745599} gives
upper and lower bounds,
\[
  h\bigl(\f^n(Q)\bigr) = d^n\bigl(h(Q)+O(1)\bigr).
\]
\par
For maps of degree~$1$, the middle inequality in~\eqref{eqn:hfnQdnhQ}
yields the stronger estimate
\[
  h\bigl(\f^n(Q)\bigr) \le h(Q) + Cn.
\]
Tracing through the proofs in this paper, this would give an
exponential improvement in our results. For example, consider a linear
map~$\f(z)=az$ with~$a\in\QQ^*$, so~$m_p(\f,1)$ is the order of~$a$ in
the multiplicative group~$\FF_p^*$. Then in place of~\eqref{eqn:sumQ}
we would obtain
\begin{equation}
  \label{eqn:mrs1}
  \sum_{p~\text{prime}} \frac{\log p}{pm_p(\f,1)^s} \le \frac{2}{s}+O(1),
\end{equation}
and this would allow us to replace Theorem~\ref{thm:densityQ}(b) with
\begin{equation}
  \label{eqn:mrs2}
  \densityinf\bigl\{p : m_p(\f,1)\ge p^\e\bigr\} \ge 1 - 2\e.
\end{equation}
However, we will not pursue the degree~$1$ case, because
both~\eqref{eqn:mrs1} and~\eqref{eqn:mrs2} are special cases of the
general results on algebraic groups proven in~\cite{MR1395936}, see in
particular~\cite[equation~(3)]{MR1395936} and the remark
following~\cite[Theorem~4.2]{MR1395936}.
\end{remark}

\begin{proposition}
\label{prop:gDgAgBhAhB}
Let~$K/\QQ$ be a number field and
let
\[
  \a_0,\ldots,\a_N,\b_0,\ldots,\b_N\in K
\]
be elements of~$K$
with at least one~$\a_i$ and at least one~$\b_i$ nonzero.
Define fractional ideals
\[
  \gA=(\a_0,\dots,\a_N),\quad
  \gB=(\b_0,\dots,\b_N),\quad
  \gD=(\a_i\b_j-\a_j\b_i)_{0\le i<j\le N}.
\]
Also let~$A=[\a_0,\dots,\a_N]\in\PP^N(K)$
and~$B=[\b_0,\dots,\b_N]\in\PP^N(K)$, and assume that~$A\ne B$. Then
\[
  \frac{1}{[K:\QQ]}
  \log\left(\frac{\Norm\gD}{\Norm\gA\cdot\Norm\gB} \right)
  \le h(A) + h(B) +O(1),
\]
where the~$O(1)$ depends only on~$N$. \textup(Here~$h$ is the 
absolute logarithm height on~$\PP^N(\Qbar)$.\textup)
\end{proposition}
\begin{proof}
We give a proof using the machinery of heights relative to closed
subschemes developed in~\cite{MR919501}, although for the specific
result that we need, some readers may prefer to write down a direct proof.
We assume that the absolute values on~$K$
have been normalized so as to obtain the absolute height,
i.e., the height of a point~$P=[x_0,\ldots,x_N]$
is given by
$h(P)=\sum_{v\in M_K}\max_i\{-v(x_i)\}$.
\par
Since~$A\ne B$, there are indices~$i$ and~$j$ such
that~$\a_i\b_j\ne\a_j\b_i$. Relabeling the coordinates, we may assume
without loss of generality that~$\a_0\b_1\ne\a_1\b_0$. We define 
subvarieties~$D$ and~$\D$ of~$\PP^N\times\PP^N$ by
\[
  D = \{ x_0y_1=x_1y_0\}
  \qquad\text{and}\qquad
  \D = \{ x_iy_j=x_jy_i~\text{for all $i$ and~$j$} \}.
\]
Thus~$\D$ is the diagonal of~$\PP^N\times\PP^N$, while~$D$ is a
divisor of type~$(1,1)$. In particular, if we let~$\pi_1$ and~$\pi_2$
be the projections~\text{$\PP^N\times\PP^N\to\PP^N$} and let~$H$ be a
hyperplane in~$\PP^N$, then~$D$ is linearly equivalent
to \text{$\pi_1^*H+\pi_2^*H$}.
\par
We also observe that~$\D\subset D$. It follows from~\cite{MR919501}
that
\begin{equation}
  \label{eqn:lDPv}
  \l_\D(P,v) \le \l_D(P,v) + O_v(1)
  \quad\text{for all 
     $P\in\bigl((\PP^N\times\PP^N)\setminus|D|\bigr)(K)$.}
\end{equation}
Here~$O_v(1)$ denotes an~$M_K$-bounded function (i.e., it is bounded by
an~$M_K$-constant) in the sense of Lang~\cite{LangDG}. Note that our
assumption on~$A$ and~$B$ ensures that~$(A,B)$ is not in the support
of~$D$. We also note that since~$D$ is an effective divisor, the local
height~$\l_D$ is bounded below by an~$M_K$-constant for all~$P$ not in
the support of~$D$. Hence evaluating~\eqref{eqn:lDPv} at~$P=(A,B)$ and
summing over all nonarchimedean~$v$ yields
\begin{align}
  \label{eqn:lDABvlehAhB}
  \sum_{v\in M_K^0} \l_\D\bigl((A,B),v\bigr) 
  &\le \sum_{v\in M_K^0} \bigl(\l_D\bigl((A,B),v\bigr) + O_v(1)\bigr) \notag\\*
  &\le \sum_{v\in M_K} \l_D\bigl((A,B),v\bigr) + O(1) \notag\\*
  &= h_D\bigl((A,B)\bigr) + O(1) \notag\\
  &= h_{\pi_1^*H+\pi_2^*H}\bigl((A,B)\bigr) + O(1) \notag\\
  &= h(A) + h(B) + O(1).
\end{align}
\par
It remains to compute the local height relative to the
diagonal. (In~\cite{MR919501},~$\l_\D$ is called an arithmetic distance
function.)  From~\cite{MR919501} and knowledge of the generators of the ideal
defining~$\D$, we see that a representative local height function
for~$\D$ is given by
\begin{equation}
  \label{eqn:lDABv}
  \l_\D\bigl((A,B),v\bigr) =
   \min_{0\le i<j\le N} v(\a_i\b_j-\a_j\b_i)
   - \min_{0\le i\le N} v(\a_i) - \min_{0\le i\le N} v(\b_i).
\end{equation}
Next we observe that for any nonzero ideal~$\gC=(\g_1,\ldots,\g_n)$,
we have
\[
  \frac{1}{[K:\QQ]}\log\Norm\gC = \sum_{v\in M_K^0} \min_{1\le i\le n} v(\g_i).
\]
(The $1/[K:\QQ]$ in front comes from the way that we have normalized
the valuations in order to simplify the formula for the absolute height.)
Hence when we sum~\eqref{eqn:lDABv} over all nonarchimedean places
of~$K$, we obtain
\[
  \sum_{v\in M_K^0} \l_\D\bigl((A,B),v\bigr) 
  = \frac{1}{[K:\QQ]}(\log\Norm\gD - \log\Norm\gA - \log\Norm\gB).
\]
Substituting this into~\eqref{eqn:lDABvlehAhB} yields the desired result.
\end{proof}

We conclude this section with an elementary result saying that every
point in~$\PP^N(K)$ has integral homogeneous coordinates that are
almost relatively prime.

\begin{lemma}
\label{lemma:intbasis}
Let~$K/\QQ$ be a number field and let~$R_K$ be the ring of integers
of~$K$. There is an integral ideal~$\gC=\gC(K)$ 
so that every~$P\in\PP^N(K)$ can be written using
homogeneous coordinates
\[
  P = [\a_0,\a_1,\ldots,\a_N]
\]
satisfying
\[
   \a_0,\ldots,\a_N\in R_K
  \qquad\text{and}\qquad
  (\a_0,\dots,\a_N)\bigm|\gC.
\]
\end{lemma}
\begin{proof}
Fix integral ideals~$\ga_1,\ldots,\ga_h$ that are representatives for
the ideal class group of~$R_K$. Given a point~$P\in\PP^N(K)$, choose
any homogeneous coordinates~$P=[\b_0,\ldots,\b_N]$. Multiplying the
coordinates by a constant, we may assume that~$\b_0,\ldots,\b_N\in
R_K$. The ideal generated by~$\b_0,\ldots,\b_N$ differs by a principal
ideal from one
of the representative ideals, say
\[
  (\g) (\b_0,\ldots,\b_N) =  \ga_j
  \quad\text{for some $\g\in K^*$.}
\]
Then each~$\g\b_i\in\ga_j\subset R_K$, so if we set~$\a_i=\g\b_i$,
then
\[
  P=[\a_0,\ldots,\a_N]\quad\text{with}\quad
  \a_0,\ldots,\a_N\in R_K\quad\text{and}\quad
  (\a_0,\ldots,\a_N) = \ga_j.
\]
Hence if we let~$\gC$ be the integral
ideal~$\gC=\ga_1\ga_2\cdots\ga_h$, then~$\gC$ depends only on~$K$, and
for any point~$P$ we have shown how to find homogeneous coordinates
in~$R_K$ such that the ideal generated by the coordinates
divides~$\gC$. 
\end{proof}

\section{An ideal characterization of orbit size}
\label{section:orbitsz}

In this section we estimate the size of a certain ideal having the
property that its prime divisors are the primes with~$m_\gp\le m$.

\begin{proposition}
\label{prop:gDm}
With notation as in Section~$\ref{section:setup}$ and
Theorem~$\ref{thm:nfdensity}$, there are a constant~$C=C(K,V,\f)$ and
a finite set of exceptional primes $\Scal=\Scal(K,V,\f)$ so that for
every~$m\ge1$ there exists a nonzero integral ideal~$\gD(m)$
satisfying the following two conditions\textup:
\begin{parts}
\Part{(i)}
For~$\gp\notin\Scal$ we have
$m_\gp \le m$ if and only if $\gp\mid\gD(m)$.
\vspace{5pt}
\Part{(ii)}
$\log\log\Norm\gD(m) \le C m$.
\end{parts}
\end{proposition}

\begin{remark}
If~$V$ is projective and~$\f$ is finite of degree~$d\ge2$,
then the following more precise version
of~\textup{(ii)} holds:
\[
  \log\log\Norm\gD(m) \le (\log d)m+C\log m.
\]
\end{remark}

\begin{proof}
By definition, $m_\gp(\f,P)$ is the smallest value of~$m$ such that there
exist~$r\ge1$ and~$s\ge0$ satisfying
\[
  r+s=m\qquad\text{and}\qquad \f^{r+s}(P)\equiv\f^s(P)\pmod{\gp}.
\]
Notice that~$s$ is the length of the tail and~$r$ is the length of the
cycle in the orbit~$\Orbit_{\tilde\f_\gp}(\tilde P\bmod \gp)$.
\par
We let~$\gC$ be the ideal described in
Lemma~\ref{lemma:intbasis}. Then for 
each~$n\ge0$ we can write
\[
  \f^n(P) = [A_0(n),A_1(n),\dots,A_N(n)]
\]
with~$A_i(n)\in R_K$ and such that the ideal
\[
  \gA(n) := \bigl(A_0(n),\dots,A_N(n)\bigr)
  \quad\text{divides the ideal~$\gC$.}
\]
\par
It follows that for all primes~$\gp\notdivide\gC$ we have
\begin{align*}
  \f^{r+s}(P)\equiv\f^s(P)&\pmod{\gp}  \\*
  &\quad\Longleftrightarrow\quad
   A_i(r+s)A_j(s) \equiv A_i(s)A_j(r+s)\pmod{\gp} \\*
   &\omit\hfil\quad\text{for all $0\le i<j\le N$.}
\end{align*}
Hence if we define ideals~$\gB(r,s)$ by
\[
  \gB(r,s) = \bigl(A_i(r+s)A_j(s) 
            - A_i(s)A_j(r+s)\bigr)_{0\le i<j\le N}
\]
and define~$\gD(m)$ to be the product
\[
  \gD(m) = \prod_{\substack{r\ge1,\,s\ge0\\r+s=m\\}} \gB(r,s),
\]
then for all primes~$\gp\notdivide\gC$ we have
\[
  m_\gp(\f,P)\le m \quad\Longleftrightarrow\quad
  \gp\mid\gD(m).
\]
Thus that~$\gD(m)$ has property~(i). Further, the assumption that~$P$
has infinite~$\f$-orbit tells us that
\[
  \f^{r+s}(P)\ne\f^s(P)\quad\text{for all~$r\ge1$ and~$s\ge0$,}
\]
so~$\gD(m)\ne0$.  
\par
It remains to estimate the norm of~$\gD(m)$.  We apply
Proposition~\ref{prop:gDgAgBhAhB}, which with our notation says that
\[
  \frac{1}{[K:\QQ]}
  \log\frac{\Norm\gB(r,s)}{\Norm\gA(r+s)\Norm\gA(s)}
  \le h\bigl(\f^{r+s}(P)\bigr) + h\bigl(\f^s(P)\bigr) + O(1).
\]
Using the fact that~$\Norm\gA(r+s)$ and~$\Norm\gA(s)$ are smaller
than~$\Norm\gC$, which only depends on~$K$, we find that
\[
  \frac{1}{[K:\QQ]}
  \log\Norm\gB(r,s)
  \le h\bigl(\f^{r+s}(P)\bigr) + h\bigl(\f^s(P)\bigr) + O(1).
\]
Next we apply Proposition~\ref{prop:hfnQlednhQC} to estimate the heights,
which gives
\[
  \log\Norm\gB(r,s)
  \le Cd^{r+s},
\]
where~$C=C(K,V,\f,P)$ and~$d=d(V,\f)\ge2$. The key point is that neither~$C$
nor~$d$ depends on~$r$ or~$s$.
\par
The ideal~$\gD(m)$ is a product of various~$\gB(r,s)$ ideals, so we obtain
\[
  \log\Norm\gD(m) 
  = \sum_{\substack{r\ge1, s\ge0\\r+s=m}}\log\Norm\gB(r,s)
  \le \sum_{\substack{r\ge1, s\ge0\\r+s=m}} Cd^{r+s}
  \le Cmd^{m+1} \le Cd^{2m}.
\]
Taking one more logarithm yields
\[
  \log\log\Norm\gD(m) \ll m,
\]
where the implied constant is independent of~$m$.
\end{proof}

An immediate corollary of Proposition~\ref{prop:gDm} is a weak
lower bound for~$m_\gp$.

\begin{corollary}
\label{cor:mpgecloglogp}
With notation as in Section~$\ref{section:setup}$, there is a
constant~$C=C(K,V,\f,P)$ so that
\[
  m_\gp(\f,P) \ge C \log\log\Norm\gp
  \quad\text{for all primes $\gp$.}
\]
\end{corollary}
\begin{proof}
For each~$m\ge1$, let~$\gD(m)$ be the nonzero the ideal described in
Proposition~\ref{prop:gDm}. Then for all but finitely many primes~$\gp$
we have
\[
  \gp\mid\gD(m_\gp)
  \qquad\text{and}\qquad
  \log\log\Norm\gD(m_\gp) \le C m_\gp.
\]
It follows that~$\Norm\gD(m_\gp)\ge\Norm\gp$, which gives the desired result
for~$\gp$. Adjusting the constant to deal with the finitely many excluded
primes completes the proof.
\end{proof}

\section{An analytic estimate}
\label{section:analyticest}

We now prove the key analytic estimate required for the proof of
Theorem~\ref{thm:nfdensity}.  This analytic result is a dynamical
analog of a theorem of Romanoff~\cite{MR1512916}, see
also~\cite{ErdosTuran1,ErdosTuran2,Landau,MR1395936}.  The proof,
especially insofar as we obtain an explicit dependence on~$s$, follows
the proof in~\cite{MR1395936}.

\begin{theorem}
\label{thm:analyticresult}
With notation as in Theorem~$\ref{thm:nfdensity}$, let~$\l\ge1$. Then
there is a constant~$C=C(K,V,\f,\l)$ so that
\begin{equation}
  \label{eqn:pmpsum1}
  \sum_{\gp\in\Spec R_K} 
  \frac{\log \Norm\gp}{\Norm\gp\cdot e^{sm_\gp^\l}}
  \le \frac{C}{s^{1/\l}}
  \quad\text{for all $s>0$.}
\end{equation}
\end{theorem}
\begin{proof}
To ease notation, define functions~$g(t)$ and~$G(t)$ by
\begin{equation}
  \label{eqn:gtGt1}
  g(t) = \frac{\log t}{t}
  \quad\text{and}\quad
  G(t) = e^{-st^\l}.
\end{equation}
We use Abel summation to rewrite the series~$S(\f,P,\l,s)$ 
in~\eqref{eqn:pmpsum1} as follows:
\begin{align} 
  \label{eqn:SfadeAbel1}
  S(\f,P,\l,s) 
  &=   \sum_{\gp\in\Spec R_K} 
       \frac{\log \Norm\gp}{\Norm\gp\cdot e^{sm_p^\l}} \notag\\
  &= \sum_{\gp\in\Spec F_K} g(\Norm\gp)G(m_\gp) \notag\\
  &= \sum_{m\ge1} 
    \Bigl(\sum_{\substack{\gp\in\Spec R_K\\ m_\gp = m\\}} 
             g(\Norm\gp)G(m)\Bigr) \notag\\
  &= \sum_{m\ge1} G(m) 
     \Bigl(
       \sum_{\substack{\gp\in\Spec R_K\\ m_\gp \le m\\}} g(\Norm p) 
       - \sum_{\substack{\gp\in\Spec R_K\\ m_p \le m-1\\}} g(\Norm p) 
     \Bigr) \notag\\
  &= \sum_{m\ge1} \bigl( G(m) - G(m+1) \bigr)
        \sum_{\substack{\gp\in\Spec R_K\\ m_\gp \le m\\}} g(\Norm\gp).
\end{align}

The mean value theorem gives
\begin{align*}
  G(m)-G(m+1)
  \le\sup_{m<\theta<m+1} -G'(\theta)
  &= \sup_{m<\theta<m+1} s\l\theta^{\l-1}e^{-s\theta^\l}\\
  &= s\l m^{\l-1}e^{-sm^\l}.
\end{align*}
Substituting into~\eqref{eqn:SfadeAbel1} yields
\begin{equation}
  \label{eqn:Sfade2}
  S(\f,\a,\l,s)  \le 
   \sum_{m\ge1} s\l m^{\l-1}e^{-sm^\l}
        \sum_{\substack{\gp\in\Spec R_K\\ m_p \le m\\}} 
        \frac{\log \Norm\gp}{\Norm\gp}.
\end{equation}

To deal with the inner sum, we use two results.  The first,
Proposition~\ref{prop:gDm}, was proven earlier. The second is as follows.

\begin{lemma}
\label{lemma:sumlogp}
Let~$K/\QQ$ be a number field. 
There are constants~$c_1$ and~$c_2$, depending only on~$K$, so that
for all integral ideals~$\gD$ we have
\[
  \sum_{\gp|\gD} \frac{\log\Norm\gp}{\Norm\gp}
  \le c_1\log\log\Norm\gD + c_2.
\]
\end{lemma}
\begin{proof}
This is a standard result. See for
example~\cite[Corollary~2.3]{MR1395936} for a derivation and an
explicit value for~$c_1$.
\end{proof}

Using the two lemmas, we obtain the bound
\begin{align*}
  S(\f,&\a,\l,s)  \\*
   &\le \sum_{m\ge1} s\l m^{\l-1}e^{-sm^\l}
        \sum_{\substack{\gp\in\Spec R_K\\ m_p \le m\\}} 
        \frac{\log \Norm\gp}{\Norm\gp}
    \quad\text{from \eqref{eqn:Sfade2},} \\
   &= \sum_{m\ge1} s\l m^{\l-1}e^{-sm^\l}
        \sum_{\substack{\gp\in\Spec R_K\\ \gp|\gD(m) \\}} 
        \frac{\log \Norm\gp}{\Norm\gp}
    \quad\text{from Proposition~\ref{prop:gDm}(i),} \\
   &\le \sum_{m\ge1} s\l m^{\l-1}e^{-sm^\l}
        \bigl(c_1\log\log\Norm\gD(m)+c_2\bigr)
     \quad\text{from Lemma~\ref{lemma:sumlogp}} \\
   &\le Cs\l \sum_{m\ge1} m^\l e^{-sm^\l} 
    \quad\text{from Proposition~\ref{prop:gDm}(ii).}
\end{align*}
It remains to deal with this last series.  If~$\l=1$, then we can explicitly
evaluate the series, but this is not possible for general values
of~$\l$.  (For example, if~$\l=2$, then it is more-or-less a theta
function). Instead we use the following elementary estimate.

\begin{lemma}
\label{lemma:msum}
Fix~$\l>0$ and~$\m\ge 0$. There is a constant~$C=C(\l,\mu)$ so that
\[
  \sum_{m=1}^\infty m^\mu e^{-s m^\l} \le Cs^{-(\mu+1)/\l}
  \qquad\text{for all $s>0$.}
\]
\end{lemma}
\begin{proof}
We estimate
\begin{align*}
  \sum_{m=1}^\infty m^\mu e^{-s m^\l}
  &\le \int_0^\infty t^\mu e^{-s t^\l} \,dt \\
  &= s^{-(\mu+1)/\l}\int_0^\infty u^\mu e^{-u^\l}\,dt
    \quad\text{letting $u = s^{1/\l}t$.}
\end{align*}
The integral converges and is independent of~$s$.
\end{proof}

Applying Lemma~\ref{lemma:msum} with~$\mu=\l$ and substituting in
above yields
\[
  S(\f,\a,\l,s) \le C_1(K,V,\f) s\l \cdot C_2(\l)s^{-(\l+1)/\l}
    = C_3(K,V,\f,\l) s^{-1/\l}.
\]
This is the desired result.
\end{proof}

\section{Proof of Theorem~\ref{thm:nfdensity}}
\label{section:proofofmainthm}

We now use the analytic estimate provided by
Theorem~\ref{thm:analyticresult} to prove our main density results.

\begin{proof}[Proof of Theorem~\ref{thm:nfdensity}] 
(a) For any~$0<\g<1$ we let
\[
  \Pcal_\g = \bigl\{\gp\in\Spec(R_K) : m_\gp \le (\log\Norm\gp)^\g\bigr\}.
\]
Then for all~$s>0$ we have
\begin{align}
  \label{eqn:Csggesum}
  \frac{C}{s^{\g}}
  &\ge \sum_{\gp\in\Spec R_K} 
     \frac{\log\Norm\gp}{\Norm\gp\cdot e^{sm_\gp^{1/\g}}}
    &&\text{from Theorem~\ref{thm:analyticresult} with $\l=1/\g$,} \notag\\
  &\ge \sum_{\gp\in\Pcal_\g} 
     \frac{\log\Norm\gp}{\Norm\gp\cdot e^{sm_\gp^{1/\g}}} \notag\\
  &\ge \sum_{\gp\in\Pcal_\g} 
     \frac{\log\Norm\gp}{\Norm\gp\cdot e^{s\log\Norm\gp}}
    &&\text{by the definition of $\Pcal_\g$,} \notag\\
  &=  \sum_{\gp\in\Pcal_\g} 
     \frac{\log\Norm\gp}{(\Norm\gp)^{1+s}}.
\end{align}
Hence
\begin{align*}
  \densitysup(\Pcal_\g)
  &= \limsup_{s\to1^+} (s-1)\sum_{\gp\in\Pcal_\g} 
     \frac{\log\Norm\gp}{(\Norm\gp)^{s}}
   &&\text{by definition of upper density,} \\
  &= \limsup_{s\to0^+} s\sum_{\gp\in\Pcal_\g} 
     \frac{\log\Norm\gp}{(\Norm\gp)^{s+1}} 
    &&\text{replacing $s$ by $s+1$,} \\
  &\le \limsup_{s\to0^+} Cs^{1-\g}
    &&\text{from \eqref{eqn:Csggesum},} \\
  &= 0
    &&\text{since $\g<1$.}
\end{align*}
Since the density is always nonnegative, this proves
that~$\density(\Pcal_\g)=0$. This is equivalent to 
Theorem~\ref{thm:nfdensity}~(a),
which asserts that the complement of~$\Pcal_\g$ has density~$1$.
\par\noindent(b)
The proof is similar. Let~$\e>0$ and define
\[
  \Pcal_\e = \bigl\{\gp\in\Spec(R_K) : m_\gp \le \e\log\Norm\gp\bigr\}.
\]
Applying Theorem~\ref{thm:analyticresult} with $\l=1$ and using the
definition of~$\Pcal_\e$, we estimate
\begin{align*}
  \label{eqn:Csggesum1}
  \frac{C}{s}
  \ge \sum_{\gp\in\Spec R_K} 
     \frac{\log\Norm\gp}{\Norm\gp\cdot e^{sm_\gp}}
  \ge \sum_{\gp\in\Pcal_\e} 
     \frac{\log\Norm\gp}{\Norm\gp\cdot e^{sm_\gp}}
  &\ge \sum_{\gp\in\Pcal_\e} 
     \frac{\log\Norm\gp}{\Norm\gp\cdot e^{s\e\log\Norm\gp}} \\
  &=  \sum_{\gp\in\Pcal_\g} 
     \frac{\log\Norm\gp}{(\Norm\gp)^{1+s\e}}.
\end{align*}
Replacing~$s$ by~$s/\e$ yields
\begin{equation*}
  \label{eqn:Csggesum1}
  \frac{C\e}{s} \ge \sum_{\gp\in\Pcal_\g} 
     \frac{\log\Norm\gp}{(\Norm\gp)^{1+s}}.
\end{equation*}
Hence
\[
  \densitysup(\Pcal_\e)
  = \limsup_{s\to0^+} s\sum_{\gp\in\Pcal_\e} 
     \frac{\log\Norm\gp}{(\Norm\gp)^{s+1}} 
  \le C\e.
\]
It follows that the complement of~$\Pcal_\e$ has lower density at
least \text{$1-C\e$}.
\end{proof}

\section{Conjectures and Experiments}
\label{section:conjandexp}

The density estimate provided by Theorem~\ref{thm:nfdensity} is
probably far from the truth. If the $\f$-orbit of a point~$P$
in~$V(\FF_\gp)$ is truly a ``random map'' from~$V(\FF_\gp)$ to itself,
then the expected orbit length~$m_\gp$ should be on the order
of~$\sqrt{\#V(\FF_\gp)}\approx \Norm\gp^{\frac12\dim V}$.
(See~\cite{MR1083961,MR0119227} for statistical properties of orbits
of random maps and~\cite{MR1094034} for the analysis of orbits of certain
polynomial maps.) Thus we might expect the quantity
\begin{equation}
  \label{eqn:pXlogpmp2}
  \frac{1}{\log X}\sum_{\Norm\gp\le X}\frac{\log\Norm\gp}{m_\gp^{2/\dim V}}
\end{equation}
to be bounded as~$X\to\infty$.
(Note that $\frac{1}{\log X}\sum_{\Norm\gp\le
X}\frac{\log\Norm\gp}{\Norm\gp}\to1$.) 
\par
We tested this guess numerically for the case of quadratic polyomial
maps on~$\PP^1_\QQ$ by computing the value of~\eqref{eqn:pXlogpmp2}
for various polynomials~$\f(z)=z^2+c$ and using values of~$X$
encompassing every~$200$th rational prime up to~$20000$.  The results
are listed in Table~\ref{table:sumlogpovermp2}.
\par
The data in Table~\ref{table:sumlogpovermp2} is ambiguous; all five
columns are increasing, but it is not clear whether they are
approaching a finite bound. We extended the calculations for~$z^2+1$
up to~$X=50000$, and the value of the sum~\eqref{eqn:pXlogpmp2}
continues to grow as shown in
Table~\ref{table:sumlogpovermp2z21}. This suggests that~$m_\gp$ grows
somewhat slower than~$\sqrt{\Norm\gp}$.

\begin{table}[t]
\begin{tabular}{|c||c|c|c|c|c|} \hline
   $X$&  $z^2-2$&  $z^2-1$&   $z^2$& $z^2+1$& $z^2+2$ \\ \hline\hline
   1223&  1.4733&    1.6042&    1.4156&  1.3539&   1.3533 \\ 
   2741&  1.7770&    1.6576&    1.5116&  1.4089&   1.4232 \\ 
   4409&  1.9864&    1.6937&    1.5711&  1.5165&   1.4911 \\ 
   6133&  2.1050&    1.6964&    1.6860&  1.6068&   1.5751 \\ 
   7919&  2.1657&    1.7330&    1.8091&  1.6314&   1.5988 \\ 
   9733&  2.2507&    1.7372&    1.8555&  1.6596&   1.6148 \\ 
   11657& 2.2868&    1.7622&    1.8825&  1.7212&   1.6722 \\
   13499& 2.3366&    1.7782&    1.9351&  1.7223&   1.7049 \\
   15401& 2.3928&    1.8822&    1.9973&  1.7307&   1.7226 \\
   17389& 2.4279&    1.9119&    2.0528&  1.7376&   1.7475 \\
   19423& 2.4551&    1.9211&    2.0726&  1.7421&   1.7582 \\
  \hline
\end{tabular}
\caption{$\frac{1}{\log X}\sum_{p\le X}\frac{\log p}{m_p^2}$ 
for various $\f(z)$}
\label{table:sumlogpovermp2}
\end{table}

\begin{table}[t]
  \begin{tabular}{|c||c|c|c|c|c|c|} \hline
  $X$ &6133  & 13499   & 21383   & 29443   & 37813   & 46447   \\ \hline
  Avg &1.6068& 1.7223& 1.7627& 1.7790& 1.8092& 1.8398 \\ \hline
  \end{tabular}
\caption{$\frac{1}{\log X}\sum_{p\le X}\frac{\log p}{m_p^2}$ for $\f(z)=z^2+1$}
\label{table:sumlogpovermp2z21}
\end{table}

Emboldened by these few experiments, we are led to make the following
plausible conjecture.

\begin{conjecture}
Let~$K/\QQ$ be a number field, let~$\f:\PP^N\to\PP^N$ be a morphism of
degree~$d\ge2$, and let~$P\in\PP^N(K)$ be a point with
infinite~$\f$-orbit. Then for every~$\e>0$ the set
\[
  \bigl\{ \gp : m_\gp(\f,P) \ge \Norm\gp^{\frac{N}{2}-\e} \bigr\}
\]
is a set of density~$1$.
\end{conjecture}

It is reasonable to pose a stronger question.  Is there
some~$\kappa>0$ so that the conjecture is true for the set of~$\gp$
satisfying
\[
  m_\gp \ge \Norm\gp^{\frac{N}{2}}\,(\log\Norm\gp)^{-\kappa}?
\]
And in the other direction, we ask whether for some constant~$C>0$,
the set of~$\gp$ satisfying~$m_\gp\ge C{\Norm\gp}^{N/2}$ is a set of
density~$0$.



\begin{thebibliography}{10}

\bibitem{MR1094034}
E.~Bach.
\newblock Toward a theory of {P}ollard's rho method.
\newblock {\em Inform. and Comput.}, 90(2):139--155, 1991.

\bibitem{ErdosTuran1}
P.~Erd\"os and P.~Turan.
\newblock Ein zahlentheoretischer {S}atz.
\newblock {\em Bull. de l'institut de Math. et M\^ec. a l'universit\'e
  Konybycheff de Tomsk}, I:101--103, 1935.

\bibitem{ErdosTuran2}
P.~Erd\"os and P.~Turan.
\newblock {\"U}ber die vereinfachung eines {L}andauschen {S}atzes.
\newblock {\em Bull. de l'institut de Math. et M\^ec. a l'universit\'e
  Konybycheff de Tomsk}, I:144--147, 1935.

\bibitem{MR1083961}
P.~Flajolet and A.~M. Odlyzko.
\newblock Random mapping statistics.
\newblock In {\em Advances in cryptology---EUROCRYPT '89 (Houthalen, 1989)},
  volume 434 of {\em Lecture Notes in Comput. Sci.}, pages 329--354. Springer,
  Berlin, 1990.

\bibitem{MR0119227}
B.~Harris.
\newblock Probability distributions related to random mappings.
\newblock {\em Ann. Math. Statist.}, 31:1045--1062, 1960.

\bibitem{MR1745599}
M.~Hindry and J.~H. Silverman.
\newblock {\em Diophantine geometry}, volume 201 of {\em Graduate Texts in
  Mathematics}.
\newblock Springer-Verlag, New York, 2000.
\newblock An introduction.

\bibitem{Landau}
E.~Landau.
\newblock Versch\"arfung eines {R}omanoffschen {S}atzes.
\newblock {\em Acta Arith.}, 1:43--62, 1935.

\bibitem{LangDG}
S.~Lang.
\newblock {\em Fundamentals of {D}iophantine {G}eometry}.
\newblock Springer-Verlag, New York, 1983.

\bibitem{MR1395936}
M.~R. Murty, M.~Rosen, and J.~H. Silverman.
\newblock Variations on a theme of {R}omanoff.
\newblock {\em Internat. J. Math.}, 7(3):373--391, 1996.

\bibitem{MR1512916}
N.~P. Romanoff.
\newblock \"{U}ber einige {S}\"atze der additiven {Z}ahlentheorie.
\newblock {\em Math. Ann.}, 109(1):668--678, 1934.

\bibitem{MR919501}
J.~H. Silverman.
\newblock Arithmetic distance functions and height functions in {D}iophantine
  geometry.
\newblock {\em Math. Ann.}, 279(2):193--216, 1987.

\end{thebibliography}
\end{document}